\documentclass[12pt]{amsart}
\title{Inequalities for the $d$th Residual Crank Moments of Overpartitions}
\author{
Ali H. Al-Saedi\\
Department of Mathematics, College of Science\\
Mustansiriya University\\
\href{mailto:alsaedi82@uomustansiriyah.edu.iq}{\nolinkurl{alsaedi82@uomustansiriyah.edu.iq}}\\
\\
Thomas Morrill\\ School of Science\\ The University of New South Wales\\ \href{mailto:t.morrill@adfa.edu.au}{\nolinkurl{t.morrill@adfa.edu.au}}\\
\\
Holly Swisher\\
Department of Mathematics\\
Oregon State University\\
\href{mailto:swisherh@math.oregonstate.edu}{\nolinkurl{swisherh@math.oregonstate.edu}}
}
\date{\today}
\usepackage{amsmath}
\usepackage{amsthm}
\usepackage{graphicx}
\usepackage{amsfonts}
\usepackage{fullpage}
\usepackage{amssymb}
\usepackage{cite}
\usepackage{hyperref}

\newtheorem{thm}{Theorem}

\newtheorem{lemma}{Lemma}

\newcommand{\ZZ}{\mathbb Z}

\newcommand{\nklist}[2]{{#1}_{1}, \ldots, {#1}_{#2}}
\newcommand{\aqprod}[3]{({#1};{#2})_{#3}}
\newcommand{\ocr}{\overline{cr}}
\newcommand{\ovr}{\frac{\aqprod{-q}{q}{\infty}}{\aqprod{q}{q}{\infty}}}

\begin{document}
\maketitle

\begin{abstract}
Two analogues of the crank function are defined for overpartitions -- the first residual crank and the second residual crank.
This suggests an exploration of crank functions defined for overpartitions whose parts are divisible by an arbitrary $d$.
We examine the positive moments of these crank functions while varying $d$ and prove some inequalities.
\end{abstract}

\section{Introduction}

The partition crank function was originally developed by Andrews and Garvan \cite{Garvan-trip, AGCrank} with the goal of giving a combinatorial proof of Ramanujan's congruence \cite{Ramanujan}
\begin{align} \label{congruence}
  p(11n+6) \equiv 0 \mod 11,
\end{align}
where Dyson's rank function was insufficient.
For a partition $\lambda$, let $\omega(\lambda)$ denote the number of occurrences of $1$ as a part in $\lambda$.
The \emph{crank} of $\lambda$ is then defined as follows.
If $\omega(\lambda) = 0$, then $cr(\lambda)$ is the largest part of $\lambda$.
Otherwise, $cr(\lambda)$ is equal to the number of parts of $\lambda$ which exceed $\omega(\lambda)$, minus $\omega(\lambda)$.

We use the common notation of $|\lambda|$ for the sum of the parts of the partition $\lambda$ and write $\lambda \vdash n$ if $|\lambda| = n$.
Let $M(m,n)$ denote the number of partitions $\lambda \vdash n$ with $cr(\lambda) = m$.
Andrews and Garvan proved \cite{AGCrank}
\begin{align*}
  M(m, 11n + 6) &= \frac{1}{11}p(11n + 6),
\end{align*}
which implies \eqref{congruence}.
The two-variable generating series for the cranks of partitions is given by\footnote{With the usual exception of $n = 1$. See Andrews and Garvan \cite{AGCrank}.}
\begin{align}             \label{Garvan-prod}
  C(z;q) = \sum_{\substack{n \geq 0 \\  m \in \ZZ}} M(m, n) z^m q^n = \frac{\aqprod{q}{q}{\infty}}{\aqprod{zq, q/z}{q}{\infty}}.
\end{align}
Here we use the $q$-Pochammer symbol,
\begin{align*}
 \aqprod{\nklist{a}{k}}{q}{\infty} = \prod_{i=0}^{\infty} (1-a_1 q^i) \cdots (1-a_k q^i).
\end{align*}
From \eqref{Garvan-prod} we can observe that
\begin{align} \label{mirror}
  M(m, n) &= M(-m, n).
\end{align}

The crank function has since been extended to the more general setting of overpartitions.
An overpartition is a non-increasing sequence of positive integers $\lambda = (\nklist{n}{k})$ in which the first occurrence of any part may be overlined.

As an example, the overpartitions of $3$ are
\begin{align*}
  (3),\  (\overline{3}),\  (2, 1),\  (2, \overline{1}),\ (\overline{2}, 1),\  (2, \overline{1}),\  (1, 1, 1),\  (\overline{1}, 1, 1),
\end{align*}
which includes the three ordinary partitions of $3$.
Corteel and Lovejoy \cite{Opartns}  put forward overpartitions as a means to combinatorially interptet $q$-hypergeometric series.

We extend the notations $|\lambda|$ and $\lambda \vdash n$ to overpartitions.
Let $P$ and $\overline{P}$ denote the sets of partitions and overpartitions, respectively, and $\overline{p}(n) = \# \{\lambda \in \overline{P} \ | \ \lambda \vdash n \}$.

Two analogs of the crank function have been defined in work of Bringmann, Lovejoy and Osburn \cite{Bringmann}.
To calculate the \emph{first residual crank} of an overpartition $\lambda$, let $\lambda'$ be the partition whose parts are the non-overlined parts of $\lambda$.
We then define $\ocr_1(\lambda) := cr(\lambda')$.
To calculate the \emph{second residual crank} of the overpartition $\lambda$, let $\lambda'$ instead be the partition whose parts are the non-overlined even parts of $\lambda$, divided by two.
We then define $\ocr_2(\lambda) := cr(\lambda')$.
As an example, $\ocr_1(4, \overline{2}, 1) = cr(4, 1) = 0$ and $\ocr_2(4, \overline{2}, 1) = cr(2) = 2$.

Let $\overline{M[1]}_k^+(n)$ and $\overline{M[2]}_k^+(n)$ denote the $k$th positive moments of these functions,
\begin{align*}
  \overline{M[1]}_k^+(n) &= \sum_{\substack{\lambda \vdash n \\ \ocr_1(\lambda) > 0}} \ocr_1(\lambda)^k\\
  \overline{M[2]}_k^+(n) &= \sum_{\substack{\lambda \vdash n \\ \ocr_2(\lambda) > 0}} \ocr_2(\lambda)^k.
\end{align*}
Here, we only sum over overpartitions whose crank is positive;
otherwise, the moments would vanish for odd $k$.
Larsen, Rust and Swisher \cite{Swish} have shown that for $n > 1$,
\begin{align*}
  \overline{M[2]}_1^+(n) < \overline{M[1]}_1^+(n).
\end{align*}
%
We extend this result both in the weight of the moment $k$, and the divisor $d$.
\begin{thm} \label{thm-please}
 For all $d,k \geq 1$,
$$
  \overline{M[d+1]}^+_k(n) \leq \overline{M[d] }^+_k(n)
$$
with equality if and only if $\overline{M[d] }^+_k(n) = 0$, which occurs if and only if $n < d$.
\end{thm}

Additionally, we give necessary and sufficient conditions on a related inequality.

\begin{thm} \label{easy} 
 For all $d,k \geq 1$,
$$
  \overline{M[d]}^+_{k}(n) \leq \overline{M[d]}^+_{k+1}(n)
$$
with equality if and only if $n < 2d$.
\end{thm}

\section{The $d$th Residual Crank}

We define the $d$th residual crank in the obvious manner.
For an overpartition $\lambda$, first let $\lambda'$ be the partition whose parts are the non-overlined parts of $\lambda$ which are divisible by $d$.
Then, divide all parts of $\lambda'$ by $d$.
We finally define $\ocr_d(\lambda) := cr(\lambda')$.

We first calculate some generating series.
Let $\overline{M[d]}(m,n)$ denote the number of overpartitions of $n$ with $d$th residual crank $m$. 
\begin{lemma} \label{d-product}
  The two variable generating series for $\overline{M[d]}(m,n)$ is given by
  \begin{align*}
    \sum_{\substack{n \geq 0 \\  m \in \ZZ}}\overline{M[d]}(m,n) z^m q^n = \ovr \frac{\aqprod{q^d, q^d}{q^d}{\infty}}{\aqprod{zq^d, q^d/z}{q^d}{\infty}}.
  \end{align*}
\end{lemma}

\begin{proof}[Proof of Lemma \ref{d-product}]
  By examining \eqref{Garvan-prod}, we see that
  \begin{align*}
    \frac{\aqprod{q^d}{q^d}{\infty}}{\aqprod{zq^d, q^d/z}{q^d}{\infty}}
  \end{align*}
  generates the nonoverlined parts of $\lambda$ which are divisible by $d$, with the exponent of $z$ equal to $\ocr_d(\lambda)$.
  Simillarly,
  \begin{align*}
    \frac{\aqprod{-q}{q}{\infty}\aqprod{q^d}{q^d}{\infty}}{\aqprod{q}{q}{\infty}}
  \end{align*}
  generates the parts of $\lambda$ which are not divisible by $d$, or are overlined.
\end{proof}

Let
\begin{align}
  \overline{M[d]}_k^+(n) &= \sum_{\substack{\lambda \vdash n \\ \ocr_d(\lambda) > 0}} \ocr_d(\lambda)^k.
\end{align}
\begin{lemma} \label{Euler-expand}
  The generating series for $\overline{M[d]}_k^+(n)$ is given by
  \begin{align*}
    \overline{C[d]}^+_k(q) :=  \sum_{n\geq 0}   \overline{M[d]}_k^+(n)  q^n = \ovr \sum_{n\geq 1}(-1)^{n+1} \frac{q^{d\tfrac{n^2+n}{2}}A_k(q^{dn})}{(1-q^{dn})^k}.
  \end{align*}
\end{lemma}

\begin{proof}
[Proof of Lemma  \ref{Euler-expand}]
Let
\begin{align*}
  G_k(z;q) := \sum_{\substack{\lambda \in P \\ cr(\lambda) > 0}} cr(\lambda)^k z^k q^{|\lambda|}.
\end{align*}
From the symmetry $M(m,n) = M(-m,n)$, we have
\begin{align*}
  \bigg(z \frac{\partial}{\partial z}\bigg)^k \overline{C[d]}(z;q) &= 
  (-q;q)_\infty \frac{(q^d;q^d)_\infty}{(q;q)_\infty} \cdot \bigg(z \frac{\partial}{\partial z}\bigg)^k C(z;q^d)\\
  &= (-q;q)_\infty \frac{(q^d;q^d)_\infty}{(q;q)_\infty} \cdot \bigg\{G_k(z; q^d) + (-1)^{k} G_k(z^{-1}; q^d) \bigg \}.
\end{align*}
Andrew, Chan, and Kim \cite{ACK} have expanded the positive moment generating series for the partition crank as
\begin{align*}
  C^+_k(q)
  = \frac{1}{\aqprod{q}{q}{\infty}} \sum_{n\geq 1}(-1)^{n+1}\frac{q^{\tfrac{n^2+n}{2}}A_k(q^{n})}{(1-q^{n})^k}
  = G(1;q),
\end{align*}
where $A_k(t)$ denotes the $k^\text{th}$ \emph{Eulerian polynomial}
\begin{align*}
  A_k(t) &= \sum_{i=0}^{k-1} A_{k,i} t^i \\
  A_{k,i} &= (i+1)A_{k-1,i} + (k-i) A_{k-1,i-1}.
\end{align*}
We see that
\begin{align*}
  \overline{C[d]}^+_k(q) =
  (-q;q)_\infty \frac{(q^d;q^d)_\infty}{(q;q)_\infty} \cdot G_k(1; q^d)=
    \ovr \sum_{n\geq 1}(-1)^{n+1} \frac{q^{d\tfrac{n^2+n}{2}}A_k(q^{dn})}{(1-q^{dn})^k}.
\end{align*}
\end{proof}

%

\section{Inequalities}

Let
\begin{align}
  h_k(q) := \sum_{n\geq 1}(-1)^{n+1} \frac{q^{\tfrac{n^2+n}{2}}A_k(q^{n})}{(1-q^{n})^k},
\end{align}
so that $\displaystyle \overline{C[d]}^+_k(q)= \ovr h_k(q^d)$.

\begin{lemma}\label{h}
  For all $k \geq 0$, the  series  $h_k(q)$ has nonnegative coefficients.
\end{lemma}

\begin{proof}
  First we expand the Eulerian polynomial $A_k(q^n)$ and the series cofficients of $\tfrac{1}{(1-q^n)^k}$ to see that
  \begin{align*}
     h_k(q)
    =
     \sum_{n\geq 1} \sum_{i=0}^{k-1} \sum_{m\geq 0} (-1)^{n+1} \binom{m+k-1}{k-1}A_{k,i} q^{\tfrac{n^2+n}{2}+in + mn}.
  \end{align*}
  We write $n=2^t s$, with $s$ odd and $t \geq 0$, in order to split the summation into its positive and a negative terms,
  \begin{multline*}
    h_k(q) = \sum_{s \text{ odd}} \sum_{m\geq 0} \sum_{i=0}^{k-1} \binom{m+k-1}{k-1}A_{k,i}q^{e_{i,s}(m,0)}\\
    - \sum_{t\geq 1}\sum_{s \text{ odd}} \sum_{m\geq 0} \sum_{i=0}^{k-1} \binom{m+k-1}{k-1}A_{k,i}q^{e_{i,s}(m,t)},
  \end{multline*}
  where
  $$
    e_{i,s}(m,t) = \frac{2^{2t}s^2+2^ts}{2}+m2^ts+i2^ts.
  $$
  
  Our goal is now to match each negative term belonging to the quadruple summation to a positive term with the same exponent belonging to the triple summation.
  Fix $i$ and $s$.
  We then define $m' = m'(m, t)$ to be
  $$
    m': = \frac{(2^{2t}-1)s + 2^t-1}{2} + (2^t-1)i + 2^t m
  $$
  It is easy to check that $e_{i,s}(m',0) = e_{i,s}(m,t)$.
  For injectivity, if we have $m_1'=m_2'$, then it follows that
  \begin{align*}
    2^{t_1}(2^{t_1}s+2i+2m_1 + 1) &= 2^{t_2}(2^{t_2}s+2i+2m_2 + 1).
  \end{align*}
  Since the parenthetical terms are odd, then $t_1=t_2$.
  This in turn implies $m_1=m_2$. Therefore, the mapping $(m,t)\mapsto (m', 0)$ is injective.

  Pairing summands in this way gives us
\begin{multline*}
  \binom{m'+k-1}{k-1}A_{k,i}q^{e_{i,s}(m',0)} -  \binom{m+k-1}{k-1}A_{k,i}q^{e_{i,s}(m,t)}\\
  = \bigg\{ \binom{m'+k-1}{k-1}-  \binom{m+k-1}{k-1}\bigg\}A_{k,i}q^{e_{i,s}(m,t)}.
\end{multline*}
 As $m' > m$, the term in braces is positive.
 Therefore $h(q)$ has nonnegative coefficients.
\end{proof}

Theorem \ref{thm-please} follows as a corollary to Lemma \ref{h}.
Recall, we seek to prove that
  for all $d,k \geq 1$,
  $$
    \overline{M[d+1]}^+_k(n) \leq \overline{M[d] }^+_k(n)
  $$
  with equality if and only if $\overline{M[d] }^+_k(n) = 0$, which occurs if and only if $n < d$.

\begin{proof}[Proof of Theorem \ref{thm-please}]
Let $\{b_n\}$ be defined by
$$
  h_k(q) = \sum_{n \geq 1} b_n q^n.
$$
Since
$$
  \overline{C[d]}^+_k(q)= \ovr h_k(q^d),
$$
and $\overline{p}(n)$ is increasing, we have
\begin{align*}
  \overline{M[d+1]}^+_{k}(n) =
  \sum_{i+(d+1)j=n}\overline{p}(i) b_j
  \leq \sum_{(i+j)+dj=n}\overline{p}(i+j) b_j
  = \overline{M[d]}^+_{k}(n),
\end{align*}
where the sum is taken over non-negative values of $i$ and $j$.
If $n<d$, then no overpartition of $n$ has parts which are divisible by $d$ or $d+1$, making both moments zero, and in particular, equal.

Conversely, if both sums are equal, then we must have $j=0$ as the only solution of $(i+j) + dj = n$, because $\overline{p}(i) b_j < \overline{p}(i+j) b_j$ for $j \geq 1$.
This implies $n < d$.
\end{proof}

We now prove Theorem \ref{easy} for completenes.
It is sufficient to show that
$
  \overline{M[d]}^+_{k}(n) = \overline{M[d]}^+_{k+1}(n)
$
if and only if $n < 2d$.

\begin{proof}

  As before, $\overline{M[d]}^+_{k}(n) = 0$ if and only if $n < d$, and thus we have equality in this case.
  
  We claim that for $0 \leq \ell <d$, we have $\overline{M[d]}^+_{k}(d + \ell) = \overline{p}(\ell)$ regardless of $k$.
  This is because overpartitions $\lambda$ of $d + \ell$ with the form $\lambda = (d) \cup \mu$, where $\mu \vdash \ell$, contribute $1$ to $\overline{M[d]}(1, n)$
  Overpartitions of $d + \ell$ not of this form\footnote{This includes $\lambda = (\overline{d}) \cup \mu$.} have $\ocr_d(\lambda) = 0$.
  Thus,
  \begin{align*}
    \overline{M[d]}^+_{k}(d + \ell) = \sum_{m \geq 1} m^k \overline{M[d]}(m, d + \ell) = 1^k \overline{M[d]}(1, d + \ell) = \overline{p}(\ell).
  \end{align*}
  
  For $n \geq 2d$, consider the overpartition $\lambda = (2d, 1, \ldots, 1) \vdash n$.
  Since $\ocr_d(\lambda) = 2$, we must have $\overline{M[d]}(2, n) \geq 1$.
  We find that
  \begin{align*}
    \sum_{m \geq 1} m^k \overline{M[d]}(m, n) < \sum_{m \geq 1} m^{k+1} \overline{M[d]}(m, n),
  \end{align*}
  since
  \begin{align*}
    2^k \overline{M[d]}(2, n) <  2^{k+1} \overline{M[d]}(2, n).
  \end{align*}
\end{proof}

\section{Further Study}

Moments of the first and second residual cranks of overpartitions are known to have a relation to moments of the Dyson rank and $M_2$-rank of overpartitions, as well as the smallest parts functions for overpartitions.

In particular, Bringmann, Lovejoy, and Osburn \cite{Bringmann} show that
\begin{align} \label{spt}
  \overline{spt}(n) &= \overline{M[1]}_1(n) - \overline{N[1]}_1(n)\\
  \overline{spt2}(n) &= \overline{M[2]}_1(n) - \overline{N[2]}_1(n).
\end{align}
We note that 
\begin{align*}
  \overline{N[2d]}_k(n) \leq \overline{N[d]}_k(n)
\end{align*}
can be shown using the same argument as in Larsen, Rust, and Swisher \cite{Swish}.
Here, $\overline{N[d]}_k(n)$ is the $k$th positive moment of the $M_d$-rank of overpartitions, which is due to the second author \cite{Morrill}.
These rank moments have the following generating series
\begin{align*}
  \sum_{n\geq 0}   \overline{N[d]}_k^+(n)  q^n
  = 2\ovr \sum_{n\geq 1}(-1)^{n+1} \frac{q^{n^2 + dn}A_k(q^{dn})}{(1+q^{dn})(1-q^{dn})^k}.
\end{align*}

It is natural to ask if there is a relation simillar to \eqref{spt} involving a $d$th smallest parts function.
However, the differences $\overline{M[d]}_1^+(n) - \overline{N[d]}_1^+(n)$ fail to be nonnegative for $d > 2$, whereas we should expect $\overline{spt[d]}(n) \geq 0$.

As part of this study, we would expect the generating series for $\overline{M[d]}_1(n)$  and $\overline{spt[d]}$ to have quasimodular properties \cite{Bringmann}.
This is presently being investigated by the second author and Aleksander Simonic.

\section{Acknowledgements}

The second author is supported by Australian Research Council Discovery Project DP160100932.


\begin{thebibliography}{ACK13}

\bibitem[ACK13]{ACK}
George~E. Andrews, Song~Heng Chan, and Byungchan Kim.
\newblock The odd moments of ranks and cranks.
\newblock {\em J. Combin. Theory Ser. A}, 120(1):77--91, 2013.

\bibitem[AG88]{AGCrank}
George~E. Andrews and F.~G. Garvan.
\newblock Dyson's crank of a partition.
\newblock {\em Bull. Amer. Math. Soc. (N.S.)}, 18(2):167--171, 1988.

\bibitem[BLO09]{Bringmann}
Kathrin Bringmann, Jeremy Lovejoy, and Robert Osburn.
\newblock Rank and crank moments for overpartitions.
\newblock {\em J. Number Theory}, 129(7):1758--1772, 2009.

\bibitem[CL04]{Opartns}
Sylvie Corteel and Jeremy Lovejoy.
\newblock Overpartitions.
\newblock {\em Trans. Amer. Math. Soc.}, 356(4):1623--1635, 2004.

\bibitem[Gar88]{Garvan-trip}
F.~G. Garvan.
\newblock New combinatorial interpretations of {R}amanujan's partition
  congruences mod {$5,7$} and {$11$}.
\newblock {\em Trans. Amer. Math. Soc.}, 305(1):47--77, 1988.

\bibitem[LRS14]{Swish}
Acadia Larsen, Alexa Rust, and Holly Swisher.
\newblock Inequalities for positive rank and crank moments of overpartitions.
\newblock {\em Int. J. Number Theory}, 10(8):2115--2133, 2014.

\bibitem[Mor19]{Morrill}
Thomas Morrill.
\newblock Two families of buffered {F}robenius representations of
  overpartitions.
\newblock {\em Ann. Comb.}, 23(1):103--141, 2019.

\bibitem[Ram21]{Ramanujan}
S.~Ramanujan.
\newblock Congruence properties of partitions.
\newblock {\em Math. Z.}, 9(1-2):147--153, 1921.

\end{thebibliography}

\end{document}